\documentclass{amsart}
\usepackage{graphicx}
\usepackage{xcolor}
\usepackage{amsmath,amsthm,amstext,amssymb,amsfonts,latexsym}

\newtheorem{thm}{Theorem}[section]
\newtheorem{lem}[thm]{Lemma}
\newtheorem{cor}[thm]{Corollary}

\theoremstyle{definition}

\newtheorem{cond}[thm]{Condition}

\theoremstyle{remark}
\newtheorem{remark}[thm]{Remark}

\newtheorem*{acknow}{Acknowledgments}

\numberwithin{equation}{section}

\allowdisplaybreaks[1]      

\begin{document}
	
	\title{Self-similar solutions of $\sigma_k^{\alpha}$-curvature flow}

	\author{Shanze Gao}
	\address{Department of Mathematical Sciences, Tsinghua University, Beijing
		100084, P.R. China}
	\email{gsz15@mails.tsinghua.edu.cn}

	\author{Hui Ma}
	\email{hma@math.tsinghua.edu.cn}
	\thanks{The authors were supported in part by NSFC grant No.~11271213 and No.~11671223.}

	\date{}
	
	\begin{abstract}
		In this paper, employing a new inequality, we show that under certain curvature pinching condition, the strictly convex closed smooth self-similar solution of  $\sigma_k^{\alpha}$-flow must be a round sphere. We also obtain a similar result for the solutions of $F=-\langle X, e_{n+1}\rangle \, (*)$ with a non-homogeneous function $F$. At last, we prove that if $F$ can be compared  with $\frac{(n-k+1)\sigma_{k-1}}{k\sigma_{k}}$, then a closed strictly $k$-convex solution of  $(*)$ must be a round sphere.
	\end{abstract}
	
	\maketitle

\section{Introduction}
\label{sec:Intro}	

	Let $X:M\rightarrow \mathbb{R}^{n+1}$ be a smooth embedding of a closed, orientable $n$-dimensional manifold with $n\geq2$.
	Choose an orthonormal frame in $\mathbb{R}^{n+1}$ along $M$ such that $\{e_1, e_2,\cdots, e_n\}$ are tangent to $M$ and  $e_{n+1}$ is  the inward-pointing unit normal vector of $M$. Under such a frame,
	let $A=\{h_{ij}\}$ denote the components of the second fundamental form of $X$, then the principal curvatures $\lambda_{1},\cdots,\lambda_{n}$ of $M$ are eigenvalues of the second fundamental form $A$.
	Define
	\begin{equation*}
		\sigma_k(A)=\frac{1}{k!}\delta
		\begin{pmatrix}
			i_{1}i_{2}\cdots i_{n}\\
			j_{1}j_{2}\cdots j_{n}
		\end{pmatrix}
		h_{i_{1}j_{1}}h_{i_{2}j_{2}}\cdots h_{i_{n}j_{n}},
	\end{equation*}
	where 
	$\delta\begin{pmatrix}
	i_{1}i_{2}\cdots i_{n}\\
	j_{1}j_{2}\cdots j_{n}
	\end{pmatrix}$
	is the generalized Kronecker symbol. We use the summation convention throughout this paper unless otherwise stated. 
	For convenience, we set $\sigma_0=1$ and $\sigma_k=0$ for $k> n$.
		
	In this paper, we consider a hypersurface $M$ which satisfies the following equation
	\begin{equation}\label{1-1}
	F(A(x))=-\langle X(x), e_{n+1}(x)\rangle, \, \text{for all}~x\in M,
	\end{equation}
	where $F(A)=f(\lambda)$ is a smooth function of principal curvatures and $\langle \,  , \, \rangle$ denotes the standard Euclidean metric in $\mathbb{R}^{n+1}$.
	
	This type of equation is important for curvature flow of the following type
	\begin{equation}\label{f}
	\tilde{X}_{t}=F(A)e_{n+1}.
	\end{equation}
	Actually, if $X$ is a solution of \eqref{1-1} and $F$ is homogeneous of degree $\beta$, then 
	$$ \tilde{X}(x,t)=((\beta+1)(T-t))^{\frac{1}{1+\beta}} X(x) $$
	gives rise to the solution of \eqref{f} up to a tangential diffeomorphism \cite{m}. So in the same spirit, we call the solutions of \eqref{1-1} self-similar solutions of \eqref{f}. Moreover, for $F=H$, the solution of \eqref{1-1} is usually called self-shrinker which describes the asymptotic behavior of mean curvature flow (see \cite{h90, c-m}). 
	Huisken proved the following theorem.
	
	\begin{thm}[\cite{h90}] \label{Thm1.1}
		If $M^n$, $n\geq2$,  is a closed hypersurface in $\mathbb{R}^{n+1}$, with non-negative mean curvature $H$ and satisfies the equation 
		\begin{equation*}
		H=-\langle X, e_{n+1}\rangle,
		\end{equation*}
		then $M^n$ is a round sphere of radius $\sqrt{n}$.
	\end{thm}

	Similar to the case of mean curvature, the solution of \eqref{1-1} with $F=\sigma_n^{\alpha}$ also describes the asymptotic behavior of $\alpha$-Gauss curvature flow (see \cite{a00,a-g-l,g-n,k-l}).
	Very recently, for $F=\sigma_{n}^{\alpha}$, Brendle, Choi and Daskalopoulos proved that the solution of \eqref{1-1} is either a round sphere for $\alpha>\frac{1}{n+2}$ or an ellipsoid for $\alpha=\frac{1}{n+2}$ (see \cite{b-c-d,c-d}).
	
	In \cite{m}, McCoy considered the case which $F$ is a class of concave or convex homogeneous functions of principal curvature with degree $1$. Under certain pinching condition, he also obtained a result for higher degree homogeneous functions.
	
	For homogeneous functions with degree greater than $1$,  the convergence of $\alpha$-Gauss curvature flow is well-studied. For
	the flows \eqref{f} of convex hypersurfaces by speeds $F=H^{\alpha}$  \cite{schu}, $F=\sigma_{2}^{\alpha}$ \cite{a, a-s} and more general  
	$F$ with certain properties \cite{a-m}, 
	similar results are obtained under certain curvature pinching conditions.

	In this paper we first consider the self-similar solutions of $\sigma_k^{\alpha}$-flow for strictly convex hypersurfaces when $k\geq 2$ and $\alpha > \frac{1}{k}$ under a different type of curvature pinching condition. 
	
	Let $\lambda=(\lambda_1,\cdots, \lambda_n)$ denote the principal curvatures of $M$.
	$M$ is said to be strictly convex if $\lambda\in \Gamma_+=\{\mu\in\mathbb{R}^{n}|\mu_{1}>0,\mu_{2}>0,\cdots,\mu_{n}>0\}$ for any point in $M$. Denote
	\begin{equation*}
	\sigma_k(\lambda)=\sigma_{k}(\lambda(A))=
	\sum_{1\leq i_{1}\leq i_{2}\cdots\leq i_{k}\leq n}
	\lambda_{i_{1}}\lambda_{i_{2}}\cdots\lambda_{i_{k}}.
	\end{equation*}
	Let $\sigma_k(\lambda|i)$ denote
	the symmetric function $\sigma_k(\lambda)$ with $\lambda_i=0$ and $\sigma_k(\lambda| ij)$,  with $i\neq j$, denote
	the symmetric function $\sigma_k(\lambda)$ with $\lambda_i=\lambda_j=0$. The following two basic equalities are needed in our investigation of the $\sigma_k^{\alpha}$ self-similar solutions.
	$$\sigma_{k-1}(\lambda|i)=\frac{\partial \sigma_{k}(\lambda)}{\partial \lambda_{i}}, \quad \sigma_{k-2}(\lambda|ij)=\frac{\partial^{2} \sigma_{k}(\lambda)}{\partial \lambda_{i}\partial \lambda_{j}}.$$
	
	\begin{cond}\label{condn}
		Assume
		\begin{equation*}
		\frac{\sigma^{2}_{1}(\lambda)\sigma_{k-2}(\lambda|ij)}
		{(\alpha k-1)\sigma_{1}(\lambda)\sigma_{k-1}(\lambda|p)
			-(\alpha-1)k^{2}\sigma_{k}(\lambda)}\in[1, 1+\delta]
		\end{equation*}
		holds for all $1\leq p \leq n, 1\leq i<j \leq n$, where 
		\begin{equation}\label{eq:delta}
			\delta=\left\{
			\begin{aligned}
				&\frac{3}{n-1}, &\text{if}~k=2,\\
				&\frac{\sqrt{n^{2}+8n-8}+2-n}{2(n-1)}, &\text{if}~3\leq k\leq n-1.
			\end{aligned}\right.
		\end{equation}
	\end{cond}
	
	Our main result can be stated as follows.
	
	\begin{thm}\label{Thm1.4}
		Let $M$ be a closed strictly convex hypersurface in $\mathbb{R}^{n+1}$ with $n\geq 2$. If $F=\sigma_{k}^{\alpha}$ with $2\leq k\leq n-1$, $\alpha>\frac{1}{k}$ and Condition \ref{condn} holds, then the solution of \eqref{1-1} is a round sphere. 
	\end{thm}
	
	\begin{remark}
		If $M$ is totally umbilical, then  Condition \ref{condn} leads to
		$$\frac{\sigma^{2}_{1}(\lambda)\sigma_{k-2}(\lambda|ij)}
		{(\alpha k-1)\sigma_{1}(\lambda)\sigma_{k-1}(\lambda|p)-
			(\alpha-1)k^{2}\sigma_{k}(\lambda)}=\frac{n}{n-1}\in[1, 1+\delta],$$
		which implies that Condition \ref{condn} is a pinching condition for principal curvatures.
	\end{remark}

	\begin{remark}	
		In the case of $F=\sigma_{2}$, 
		Condition \ref{condn} satisfies if $3\lambda_{\mathrm min}\geq \lambda_{\mathrm max}$, where $\lambda_{\mathrm min}$ and $\lambda_{\mathrm max}$ are the minimum principal curvature and maximum principal curvature, respectively. 
	\end{remark}
		
	\begin{remark} 
	Brendle, Choi and Daskalopoulos obtained better result for the case of $F=\sigma_n^{\alpha}$ in \cite{b-c-d,c-d}, so we omit this case in the statement of Theorem \ref{Thm1.4}.
	\end{remark}
	
	Somewhat surprisingly, the above argument enables us to discuss the solution of \eqref{1-1} with a non-homogeneous function $F=\sum_{l=1}^{n}a_{l}\sigma_{l}$, where $a_{l}$ are nonnegative constants with  $\sum_{l=2}^{n}a_{l}>0$. Thus we obtain the following result.
	\begin{thm}\label{Thm1.7}
		Let $M$ be a closed strictly convex hypersurface in $\mathbb{R}^{n+1}$ with $n\geq 2$. If $F=\sum_{l=1}^{n}a_{l}\sigma_{l}$ and the condition $\lambda_{min}\geq \Theta\lambda_{max}$
		holds, where $0<\Theta \leq 1$ is a constant depending on $n$,
		then, the solution of \eqref{1-1} is a round sphere. 
	\end{thm}
	
	\begin{remark}
		For $F=\sum_{l=1}^{n}a_{l}\sigma_{l}^{\alpha_{l}}$ with $\alpha_l >\frac{1}{l}$, under suitable pinching condition, the solution of \eqref{1-1} is also a round sphere. The proof is similar to the proof of Theorem \ref{Thm1.7}.
	\end{remark}
	
	For $F=\frac{\sigma_{k-1}}{\sigma_{k}}$ (which is used in \cite{g-l}), the solution of \eqref{1-1} can be characterized as follows when $M$ is strictly $k$-convex.  A hypersurface $M$ in $\mathbb{R}^{n+1}$ is strictly $k$-convex, if  $\lambda(x) \in\Gamma_{k}=\{\mu\in\mathbb{R}^{n}|\sigma_{1}(\mu)>0,\cdots,\sigma_{k}(\mu)>0\}$ for all $x\in M$. Obviously, $\Gamma_{+}=\Gamma_{n}\subset\Gamma_{n-1}\subset\cdots\subset\Gamma_{1}$. 
	
	\begin{thm}\label{inthm}
		Let $M$ be a closed strictly $k$-convex hypersurface in $\mathbb{R}^{n+1}$. Assuming $F\geq\frac{(n-k+1)\sigma_{k-1}}{k\sigma_{k}}$ or $F\leq\frac{(n-k+1)\sigma_{k-1}}{k\sigma_{k}}$, if there exists a solution of \eqref{1-1}, then $F=\frac{(n-k+1)\sigma_{k-1}}{k\sigma_{k}}$ and the solution must be a round sphere.
	\end{thm}
	
	The paper is organized as follows.  In Section \ref{sec:Prelim}, we show a new inequality of symmetric functions, which plays an important role in the proof of our main result. Some basic equations are derived in Section \ref{sec:preparation}. In Section \ref{sec:sigma_k},  we use the maximum principle to establish our main result (Theorem \ref{Thm1.4}). We devote Section \ref{sec:F non-homogeneous} to a discussion on the solution of \eqref{1-1}  with a non-homogeneous function $F$. Finally the proof of Theorem \ref{inthm} is presented in Section \ref{sec:other}.
	
\section{A new inequality of symmetric functions}
\label{sec:Prelim}
	
	In this section we show a new inequality of symmetric functions, which may have its own interest.
	
	\begin{lem}\label{Lem 2.1}
		For any $2\leq k \leq n$ and $\lambda \in \Gamma_+$, we have
		$$	\frac{1}{k(k-1)}\sigma_{1}(\lambda)-\frac{k\sigma_k (\lambda)}{(k-1)\sigma_{k-1}(\lambda)}+ \frac{(k+1)\sigma_{k+1}(\lambda)}{k\sigma_{k}(\lambda)}\geq 0. $$
		Equality occurs if and only if $\lambda_1=\lambda_2=\cdots=\lambda_n$.
	\end{lem}
	
	\begin{proof}
		Let $S_k (\lambda)$ denote the power sum of $\lambda$ defined by $S_k (\lambda)=\sum_{i=1}^n \lambda_i^k$.
		Then $2\sigma_2=S_1^2-S_2$ and $3\sigma_3=\frac{1}{2}S_1^3-\frac{3}{2}S_1 S_2+S_3$. Thus for $k=2$, we have
		\begin{align*}
			&\quad\sigma_{1}-4\frac{\sigma_2}{\sigma_1}+3\frac{\sigma_{3}}{\sigma_{2}}\\
			&=\frac{1}{\sigma_1 \sigma_2} (\sigma_{1}^2\sigma_{2}-4\sigma_{2}^2+3\sigma_{1}\sigma_{3})\\
			&=\frac{1}{\sigma_1 \sigma_2} \big( \frac{1}{2}S_{1}^2(S_{1}^2-S_{2})-(S_{1}^2-S_{2})^2+S_{1}(\frac{1}{2}S_{1}^3-\frac{3}{2}S_{1}S_{2}+S_{3}) \big)\\
			&=\frac{1}{\sigma_1 \sigma_2}(S_{1}S_{3}-S_{2}^2)\\
			&=\frac{1}{\sigma_1 \sigma_2}\big( \sum_{i<j}\lambda_i\lambda_j(\lambda_i-\lambda_j)^2 \big)\geq 0
		\end{align*}
		and equality occurs if and only if $\lambda_1=\lambda_2=\cdots=\lambda_n$.
		
		We complete the proof by induction for $k$ and assume the lemma is true for $\{2,3,\cdots,k-1\}$.
		Let
		\begin{equation*}
			f(x)=\prod_{i=1}^{n} (1-\lambda_{i}x)=\sum_{m=0}^{n}(-1)^{m}\sigma_{m}(\lambda)x^{m}.
		\end{equation*}
		Then
		\begin{equation*}
			\frac{d}{dx}f(x)=\sum_{m=1}^{n}(-1)^{m}m\sigma_{m} (\lambda)x^{m-1}.
		\end{equation*}
		On the other hand, since $\frac{d}{dx}f(x)$ is a polynomial of degree $n-1$, by Rolle's theorem, if all roots of a polynomial $f(x)$ are real and positive, then the same is true for its derivative. This leads to 
		\begin{equation*}
			\frac{d}{dx}f(x)=-\sigma_{1}(\lambda)\prod_{i=1}^{n-1} (1-\mu_{i}x)=-\sigma_{1}(\lambda)\sum_{l=0}^{n-1}(-1)^{l}\sigma_{l}(\mu)x^{l}.
		\end{equation*}
		By comparing the above two expressions, we conclude that
		\begin{equation}\label{induction}
			(m+1)\sigma_{m+1} (\lambda)=\sigma_{1}(\lambda)\sigma_{m}(\mu) \quad \text{for}\quad 0\leq m\leq n-1.
		\end{equation}
		Thus, for $2\leq k\leq n-1$, we obtain
		\begin{align*}
			\frac{(k+1)\sigma_{k+1}(\lambda)}{k\sigma_{k}(\lambda)}
			&=\frac{\sigma_{k}(\mu)}{\sigma_{k-1}(\mu)}\\
			&\geq \frac{k-1}{k}\left(\frac{(k-1)\sigma_{k-1}(\mu)}{(k-2)\sigma_{k-2}(\mu)}-\frac{\sigma_{1}(\mu)}{(k-1)(k-2)}\right)\\
			&=\frac{k-1}{k}\left(\frac{k\sigma_{k}(\lambda)}{(k-2)\sigma_{k-1}(\lambda)}-\frac{2\sigma_{2}(\lambda)}{(k-1)(k-2)\sigma_{1}(\lambda)}\right)\\
			&=\frac{(k-1)^{2}}{k(k-2)}\frac{k\sigma_{k}(\lambda)}{(k-1)\sigma_{k-1}(\lambda)}-\frac{2\sigma_{2}(\lambda)}{k(k-2)\sigma_{1}(\lambda)}\\
			&=\frac{k\sigma_{k}(\lambda)}{(k-1)\sigma_{k-1}(\lambda)}+\frac{1}{k(k-2)}\left(\frac{k\sigma_{k}(\lambda)}{(k-1)\sigma_{k-1}(\lambda)}-\frac{2\sigma_{2}(\lambda)}{\sigma_{1}(\lambda)}\right)\\
			&=\frac{k\sigma_{k}(\lambda)}{(k-1)\sigma_{k-1}(\lambda)}+\frac{1}{k(k-2)}\sum_{i=2}^{k-1}\left(\frac{(i+1)\sigma_{i+1}(\lambda)}{i\sigma_{i}(\lambda)}-\frac{i\sigma_{i}(\lambda)}{(i-1)\sigma_{i-1}(\lambda)}\right)\\
			&\geq \frac{k\sigma_{k}(\lambda)}{(k-1)\sigma_{k-1}(\lambda)}-\frac{1}{k(k-2)}\sum_{i=2}^{k-1}\frac{\sigma_{1}(\lambda)}{i(i-1)}\\
			&=\frac{k\sigma_{k}(\lambda)}{(k-1)\sigma_{k-1}(\lambda)}-\frac{1}{k(k-1)}\sigma_{1}(\lambda).
		\end{align*}
		For $k=n$, since $\sigma_{n+1}=0$, it is confirmed by the Newton-MacLaurin inequalities.
		We finish the proof by noticing all equalities occur if and only if $\lambda_1=\lambda_2=\cdots=\lambda_n$.
	\end{proof}

	\begin{remark}
		The condition $\lambda\in \Gamma_{+}$ seems necessary because for $k=2,n=3$ and $\lambda=(-1,3,3)$, we have $\sigma_{1}(\lambda)>0$ and $\sigma_{2}(\lambda)>0$ but
		$$\sigma_{1}-4\frac{\sigma_2}{\sigma_1}+3\frac{\sigma_{3}}{\sigma_{2}}=\frac{1}{\sigma_1 \sigma_2}\big( \sum_{i<j}\lambda_i\lambda_j(\lambda_i-\lambda_j)^2 \big)<0.$$
	\end{remark}
	
	The following corollary of Lemma 2.1 will be used in Section \ref{sec:sigma_k}.
	\begin{cor}\label{Cor 2.3}
		For any $2\leq k \leq n$ and $\lambda \in \Gamma_+$, we have
		$$	(k-1)\sigma_{1}(\lambda)-2k\frac{\sigma_2 (\lambda)}{\sigma_1 (\lambda)}+(k+1) \frac{\sigma_{k+1}(\lambda)}{\sigma_{k}(\lambda)}\geq 0. $$
		Equality occurs if and only if $\lambda_1=\lambda_2=\cdots=\lambda_n$.
	\end{cor}

	\begin{proof}
		Notice
		\begin{align*}
			\frac{(k+1)\sigma_{k+1}(\lambda)}{k\sigma_{k}(\lambda)}-\frac{2\sigma_{2}(\lambda)}{\sigma_{1}(\lambda)}&=\sum_{i=2}^{k}\left(\frac{(i+1)\sigma_{i+1}(\lambda)}{i\sigma_{i}(\lambda)}-\frac{i\sigma_{i}(\lambda)}{(i-1)\sigma_{i-1}(\lambda)}\right)\\
			&\geq - \sum_{i=2}^{k}\frac{\sigma_{1}(\lambda)}{i(i-1)}\\
			&=-(1-\frac{1}{k})\sigma_{1}(\lambda),
		\end{align*}
		where the inequality follows from Lemma \ref{Lem 2.1}.
	\end{proof}
	
	To finish this section, we list one well-known result (See for example \cite{a94} and \cite{Gerhardt}).

	\begin{lem}\label{Lem 2.2}
		If $W=(w_{ij})$ is a symmetric real matrix and $\lambda_{m}=\lambda_{m}(W)$ is one of its eigenvalues ($m=1,\cdots,n$).
		If $f=f(\lambda)$ is a function on $\mathbb{R}^{n}$ and $F=F(W)=f(\lambda(W))$, then for any real symmetric matrix $B=(b_{ij})$, we have the following formulae:
		\begin{itemize}
			\item[(i)]
			$\displaystyle\frac{\partial F}{\partial w_{ij}}b_{ij}=\frac{\partial f}{\partial \lambda_{p}}b_{pp},$
			\item[(ii)]
			$\displaystyle\frac{\partial^{2} F}{\partial w_{ij}\partial w_{st}}b_{ij}b_{st}=\frac{\partial^{2} f}{\partial \lambda_{p}\partial \lambda_{q}}b_{pp}b_{qq}+2\sum_{p<q}\frac{\frac{\partial f}{\partial \lambda_{p}}-\frac{\partial f}{\partial \lambda_{q}}}{\lambda_{p}-\lambda_{q}}b_{pq}b_{qp}.$
		\end{itemize}
	\end{lem}

	\begin{remark}
		In the above lemma, $\frac{\frac{\partial f}{\partial \lambda_{p}}-\frac{\partial f}{\partial \lambda_{q}}}{\lambda_{p}-\lambda_{q}}$ is interpreted as a limit if $\lambda_{p}=\lambda_{q}$.
	\end{remark}

\section{Equations of the test function}
\label{sec:preparation}
	
	In this section, we will obtain some useful equations by direct computation from the following equation
	\begin{equation}\label{2-0}
	F=-\langle X, e_{n+1}\rangle.
	\end{equation}
	
	Differentiating (\ref{2-0}) gives
	\begin{equation*}
	\nabla_{j}F=h_{jl}\langle X, e_l \rangle,
	\end{equation*}
	and
	\begin{equation}\label{2-1}
	\nabla_{i}\nabla_{j}F=h_{jli}\langle X, e_l \rangle+h_{ij}-h_{jl}h_{li}F.
	\end{equation}
	
	Then, we obtain
	\begin{equation}\label{2-2}
	\frac{\partial F}{\partial h_{ij}}\nabla_{i}\nabla_{j}F
	=\nabla_{l}F\langle X, e_l \rangle+\frac{\partial F}{\partial h_{ij}}h_{ij}-\frac{\partial F}{\partial h_{ij}}h_{jl}h_{li}F.
	\end{equation}
	
	Let $G=G(A)=g(\lambda(A))$ be a homogeneous function, called the test function in this paper. By directly calculation, we obtain
\begin{equation*}
\nabla_{i}\nabla_{j}G=\nabla_{i}\left( \frac{\partial G}{\partial h_{pq}}h_{pqj}\right)
	=\frac{\partial^2 G}{\partial h_{pq}\partial h_{st}}h_{pqj}h_{sti}
	+\frac{\partial G}{\partial h_{pq}}h_{pqji}.
	\end{equation*} 
	By Codazzi equation and Ricci identity, we obtain
	\begin{equation*}
	h_{pqji}=h_{pjqi}=h_{pjiq}+h_{mj}R_{mpqi}+h_{pm}R_{mjqi}.
	\end{equation*}
	Furthermore, using Gauss equation gives rise to
	\begin{equation*}
	h_{pqji}=h_{ijpq}+h_{mj}(h_{mq}h_{pi}-h_{mi}h_{pq})
	+h_{pm}(h_{mq}h_{ji}-h_{mi}h_{jq}).
	\end{equation*}
	Then, we have
	\begin{align*}
		\frac{\partial F}{\partial h_{ij} } \nabla_{i}\nabla_{j}G
		&=\frac{\partial F}{\partial h_{ij}}\frac{\partial^2 G}{\partial h_{pq}\partial h_{st}}h_{pqj}h_{sti}+\frac{\partial F}{\partial h_{ij}}\frac{\partial G}{\partial h_{pq}}h_{ijpq}\\
		&~+\frac{\partial F}{\partial h_{ij}}\frac{\partial G}{\partial h_{pq}}(h_{mj}(h_{mq}h_{pi}-h_{mi}h_{pq})+h_{pm}(h_{mq}h_{ji}-h_{mi}h_{jq}))\\
		&=\frac{\partial F}{\partial h_{ij}}\frac{\partial^2 G}{\partial h_{pq}\partial h_{st}}h_{pqj}h_{sti}+\frac{\partial G}{\partial h_{pq}}\left(\nabla_{p}\nabla_{q}
		F-\frac{\partial^{2} F}{\partial h_{ij}\partial h_{st}}h_{ijp}h_{stq}\right)\\
		&~+\frac{\partial F}{\partial h_{ij}}\frac{\partial G}{\partial h_{pq}}(-h_{mj}h_{mi}h_{pq}+h_{pm}h_{mq}h_{ji}).
	\end{align*}
	Moreover, using \eqref{2-1}, we obtain
	\begin{equation}\label{2-3}
	\begin{split}
	&\quad\frac{\partial F}{\partial h_{ij} } \nabla_{i}\nabla_{j}G-\nabla_{l}G\langle X, e_l \rangle\\
	&=\frac{\partial G}{\partial h_{ij}}h_{ij}\left(1-\frac{\partial F}{\partial h_{pq}}h_{pm}h_{mq} \right)+\frac{\partial G}{\partial h_{ij}}h_{jl}h_{li}
	\left(\frac{\partial F}{\partial h_{pq}}h_{pq}-F\right)\\
	&~+\left( \frac{\partial F}{\partial h_{ij}}\frac{\partial^2 G}{\partial h_{pq}\partial h_{st}}
	-\frac{\partial G}{\partial h_{ij}}\frac{\partial^{2} F}{\partial h_{pq}\partial h_{st}}\right) h_{pqj}h_{sti}.
	\end{split}
	\end{equation}
	
	The above equation is elliptic if $\frac{\partial F}{\partial h_{ij}}$ is positive definite. In fact, for $\lambda\in\Gamma_{+}$, the cases that we discuss are all elliptic.
	
	For convenience, we denote the first two terms on the right hand of the above equality by TERM I and the last term by TERM II. In fact, by Lemma \ref{Lem 2.2}, we have
	$$\text{TERM I}=\frac{\partial g}{\partial \lambda_{i}}\lambda_{i}\left(1-\frac{\partial f}{\partial \lambda_{p}}\lambda_{p}^{2}\right)+\frac{\partial g}{\partial \lambda_{i}}\lambda_{i}^{2}\left(\frac{\partial f}{\partial \lambda_{p}}\lambda_{p}-f\right)$$
	and
	\begin{align*}
	\text{TERM II}&=\left(\frac{\partial f}{\partial \lambda_{i}}\frac{\partial^{2} g}{\partial \lambda_{p} \partial \lambda_{q}}-\frac{\partial g}{\partial \lambda_{i}}\frac{\partial^{2} f}{\partial \lambda_{p} \partial \lambda_{q}}\right)h_{ppi}h_{qqi}\\
	&~+2\sum_{p<q}\left(\frac{\frac{\partial g}{\partial \lambda_{p}}-\frac{\partial g}{\partial \lambda_{q}}}{\lambda_{p}-\lambda_{q}}\frac{\partial f}{\partial \lambda_{i}}-\frac{\frac{\partial f}{\partial \lambda_{p}}-\frac{\partial f}{\partial \lambda_{q}}}{\lambda_{p}-\lambda_{q}}\frac{\partial g}{\partial \lambda_{i}}\right)h_{pqi}^{2}.
	\end{align*}

\section{the case of $F=\sigma_{k}^{\alpha}$  $(k\geq 2)$}
\label{sec:sigma_k}

	In this section, we consider the case of $F=\sigma_{k}^{\alpha}$ $(k\geq 2)$ and we choose $G=\frac{\sigma_{1}^{k}}{\sigma_{k}}$ as the  test function. By straightforward calculation, we obtain the following expressions which will be used later:
	\begin{equation}\label{f2}
		\frac{\partial^{2} f}{\partial \lambda_{p} \partial \lambda_{q}}=\alpha f\left(\frac{(\alpha-1)\sigma_{k-1}(\lambda|p)\sigma_{k-1}(\lambda|q)}{\sigma_{k}^{2}}+\frac{\sigma_{k-2}(\lambda|pq)}{\sigma_{k}}\right),
	\end{equation}
	
	\begin{equation}\label{f-}
		\frac{\frac{\partial f}{\partial \lambda_{p}}-\frac{\partial f}{\partial \lambda_{q}}}{\lambda_{p}-\lambda_{q}}=-\alpha f\frac{\sigma_{k-2}(\lambda|pq)}{\sigma_{k}},
	\end{equation}
	
	\begin{equation}\label{g1}
		\frac{\partial g}{\partial\lambda_{p}}=g\left(\frac{k}{\sigma_{1}}-\frac{\sigma_{k-1}(\lambda|p)}{\sigma_{k}}\right),
	\end{equation}
	
	\begin{equation}\label{g2}
		\begin{aligned} 
		\frac{\partial^{2} g}{\partial\lambda_{p}\partial\lambda_{q}}
		&=g\Big(\frac{k(k-1)}{\sigma_{1}^{2}}-\frac{k(\sigma_{k-1}(\lambda|p)+\sigma_{k-1}(\lambda|q))}{\sigma_{1}\sigma_{k}}\\
		&\quad-\frac{\sigma_{k-2}(\lambda|pq)}{\sigma_{k}}+\frac{2\sigma_{k-1}(\lambda|p)\sigma_{k-1}(\lambda|q)}{\sigma_{k}^{2}}\Big),
		\end{aligned}
	\end{equation}
	
	\begin{equation}\label{g-}
		\frac{\frac{\partial g}{\partial \lambda_{p}}-\frac{\partial g}{\partial \lambda_{q}}}{\lambda_{p}-\lambda_{q}}=g\frac{\sigma_{k-2}(\lambda|pq)}{\sigma_{k}}.
	\end{equation}

	We will see TERM I is non-negative and TERM II is non-negative under Condition \ref{condn}.
	
	\begin{lem}\label{Lem4.1}
		For $F=\sigma_{k}^{\alpha}$ $(k\geq 2)$ and $G=\frac{\sigma_{1}^{k}}{\sigma_{k}}$, \emph{TERM I} is non-negative 
		for $\alpha\geq \frac{1}{k}$. Moreover, it vanishes for $\alpha>\frac{1}{k}$ if and only if $\lambda_1=\lambda_2=\cdots=\lambda_n$.
	\end{lem}
	
	\begin{proof}
		Noting that $G=\frac{\sigma_{1}^{k}}{\sigma_{k}}$ is homogeneous of degree $0$, we have $\frac{\partial g}{\partial \lambda_{i}}\lambda_{i}=0$.	And, \eqref{g1} and 
	$\sum_{i=1}^{n}\sigma_{k-1}(\lambda|i)\lambda_{i}^{2}=\sigma_{1}(\lambda)\sigma_{k}(\lambda)-(k+1)\sigma_{k+1}(\lambda)$ yield
		\begin{align*}
		\frac{\partial g}{\partial \lambda_{p}}\lambda_{p}^{2}
		&=g\left(\frac{k(\sigma_{1}^{2}-2\sigma_{2})}{\sigma_{1}}-\frac{\sigma_{k-1}(\lambda|p)\lambda_{p}^{2}}{\sigma_{k}}\right)\\
		&=g\left((k-1)\sigma_{1}-2k\frac{\sigma_2}{\sigma_1}+(k+1)\frac{\sigma_{k+1}}{\sigma_{k}}\right).
		\end{align*}
		Thus,
		$$\text{TERM I}=(k\alpha-1)fg\left((k-1)\sigma_{1}-2k\frac{\sigma_2}{\sigma_1}+(k+1)\frac{\sigma_{k+1}}{\sigma_{k}}\right).$$
		For $f>0$ and $g>0$, the proof is finished by Corollary \ref{Cor 2.3}.
	\end{proof}
	
	Now, we consider TERM II and complete the proof of Theorem \ref{Thm1.1}.
	\begin{lem}
		For $F=\sigma_{k}^{\alpha}$ $(k\geq 2)$ and $G=\frac{\sigma_{1}^{k}}{\sigma_{k}}$, if $G$ attains its maximum at $x_{0}$, then, at $x_{0}$,
		\begin{align*}
		\mathrm{TERM ~ II}&=\frac{\alpha kfg}{\sigma_{1}^{3}\sigma_{k}}\Big(\sum_{i}\sum_{p\neq q}\sigma_{1}^{2}\sigma_{k-2}(\lambda|pq)(h_{pqi}^{2}-h_{ppi}h_{qqi})\\
		&~+\sum_{i}\big(-(\alpha-1)k^{2}\sigma_{k}+(\alpha k-1)\sigma_{1}\sigma_{k-1}(\lambda|i)\big)(\nabla_{i}\sigma_{1})^{2})\Big).
		\end{align*}
	\end{lem}
	
	\begin{proof}
		By Lemma \ref{Lem 2.2}, we have $\sigma_{k-1}(\lambda|p)h_{ppi}=\nabla_{i}\sigma_{k}$. Then, by \eqref{f2} and \eqref{g2}, $$\frac{\partial^{2} f}{\partial \lambda_{p} \partial \lambda_{q}}h_{ppi}h_{qqi}=\alpha f\left(\sum_{i}\frac{(\alpha-1)(\nabla_{i}\sigma_{k})^{2}}{\sigma_{k}^{2}}+\sum_{i}\sum_{p\neq q}\frac{\sigma_{k-2}(\lambda|pq)}{\sigma_{k}}h_{ppi}h_{qqi}\right)$$
		and 
		\begin{align*}
			\frac{\partial^{2} g}{\partial\lambda_{p}\partial\lambda_{q}}h_{ppi}h_{qqi}=&g\sum_{i}\Big(\frac{k(k-1)(\nabla_{i}\sigma_{1})^{2}}{\sigma_{1}^{2}}-\frac{2k\nabla_{i}\sigma_{1}\nabla_{i}\sigma_{k}}{\sigma_{1}\sigma_{k}}\\
			&-\sum_{p\neq q}\frac{\sigma_{k-2}(\lambda|pq)}{\sigma_{k}}h_{ppi}h_{qqi}+\frac{2(\nabla_{i}\sigma_{k})^{2}}{\sigma_{k}^{2}}\Big).
		\end{align*}
		
		Since $G$ attains its maximum at $x_{0}$, then $\nabla_{l}G=0$ at $x_0$ which  implies
		$\displaystyle\frac{k\nabla_{l}\sigma_{1}}{\sigma_{1}}=\frac{\nabla_{l}\sigma_{k}}{\sigma_{k}}$ at $x_{0}$.
		Thus, at $x_{0}$, $$\frac{\partial^{2} f}{\partial \lambda_{p} \partial \lambda_{q}}h_{ppi}h_{qqi}=\alpha f\left(\sum_{i}\frac{(\alpha-1)k^{2}(\nabla_{i}\sigma_{1})^{2}}{\sigma_{1}^{2}}+\sum_{i}\sum_{p\neq q}\frac{\sigma_{k-2}(\lambda|pq)}{\sigma_{k}}h_{ppi}h_{qqi}\right)$$
		and $$\frac{\partial^{2} g}{\partial\lambda_{p}\partial\lambda_{q}}h_{ppi}h_{qqi}=g\Big(\sum_{i}\frac{k(k-1)(\nabla_{i}\sigma_{1})^{2}}{\sigma_{1}^{2}}-\sum_{i}\sum_{p\neq q}\frac{\sigma_{k-2}(\lambda|pq)}{\sigma_{k}}h_{ppi}h_{qqi}\Big).$$
		
		Furthermore, using \eqref{f-}, \eqref{g1}, \eqref{g-}, at $x_{0}$, we get
		\begin{align*}
			\text{TERM II}&=\frac{\alpha kfg}{\sigma_{1}^{3}\sigma_{k}}\Big(\sum_{i}\sum_{p\neq q}\sigma_{1}^{2}\sigma_{k-2}(\lambda|pq)(h_{pqi}^{2}-h_{ppi}h_{qqi})\\
			&~+\sum_{i}\big(-(\alpha-1)k^{2}\sigma_{k}+(\alpha k-1)\sigma_{1}\sigma_{k-1}(\lambda|i)\big)(\nabla_{i}\sigma_{1})^{2})\Big).
		\end{align*}
	\end{proof}
	
	For convenience, we denote 
	$$A_{ij}=\sigma_{1}^{2}\sigma_{k-2}(\lambda|ij)$$ 
	and 
	$$B_{p}=-(\alpha-1)k^{2}\sigma_{k}+(\alpha k-1)\sigma_{1}\sigma_{k-1}(\lambda|p).$$ Then, 
	$$\text{TERM II}=\frac{\alpha kfg}{\sigma_{1}^{3}\sigma_{k}}\big(\sum_{i\neq j}  \sum_{p} A_{ij}(h_{ijp}^{2}-h_{iip}h_{jjp})+\sum_{p} B_{p}(\nabla_{p}\sigma_{1})^{2})\big).$$

	
	\begin{lem}\label{Lem 4.3}
		Let $M$ be a closed strictly convex hypersurface in $\mathbb{R}^{n+1}$ with $n\geq 2$ satisfying Condition \ref{condn}. For $F=\sigma_{k}^{\alpha}$ $(k\geq 2)$ and $G=\frac{\sigma_{1}^{k}}{\sigma_{k}}$, if $G$ attains its maximum at $x_{0}$, then, at $x_{0}$, \emph{TERM II} is non-negative.
	\end{lem}
	
	\begin{proof}
		It  suffices to check if  $\sum_{i\neq j}\sum_p A_{ij}(h_{ijp}^{2}-h_{iip}h_{jjp})
		+\sum_{p}B_{p}(\nabla_{p}\sigma_{1})^{2}$ is non-negative.
		Firstly, we notice		
		\begin{align*}
		&~\sum_{i\neq j}\sum_p A_{ij}(h_{ijp}^{2}-h_{iip}h_{jjp})
		+\sum_{p}B_{p}(\nabla_{p}\sigma_{1})^{2}\\
		&=\sum_{i\neq j}A_{ij}(h_{iji}^{2}+h_{ijj}^{2})
		+\sum_{\neq}A_{ij}h_{ijp}^{2}-\sum_{i\neq j}A_{ij}(h_{iii}h_{jji}+h_{iij}h_{jjj})\\
		&~-\sum_{\neq}A_{ij}h_{iip}h_{jjp}
		+\sum_{p}B_{p}(\sum_{i}h_{iip}^{2}+\sum_{i\neq j}h_{iip}h_{jjp})\\
		&=2\sum_{i\neq j}A_{ij}h_{iij}^{2}+\sum_{\neq}A_{ij}h_{ijp}^{2}-2\sum_{i\neq j}A_{ij}h_{iii}h_{jji}-\sum_{\neq}A_{ij}h_{iip}h_{jjp}
		+\sum_{i}B_{i}h_{iii}^{2}\\
		&~+\sum_{i\neq p}B_{p}h_{iip}^{2}+\sum_{i\neq j} (B_{i}h_{iii}h_{jji}+B_{j}h_{iij}h_{jjj}) +\sum_{\neq}B_{p}h_{iip}h_{jjp}\\
		&= \sum_{i}B_{i}h_{iii}^{2}+\sum_{i\neq j}\left( 2A_{ij}+B_{j} \right)h_{iij}^{2}+\sum_{\neq}A_{ij}h_{ijp}^{2}
		+2\sum_{i\neq j}\left(-A_{ij}+B_{i} \right) h_{iii}h_{jji}\\
		&~+\sum_{\neq}\left(-A_{ij}+ B_{p} \right)h_{iip}h_{jjp},
		\end{align*}
		where $\neq$ represents $i, j, p$ are pairwise distinct.
		
		Now, 
		we estimate the lower bounds of the last two terms. For fixed $i,j$ and $p$, we have
		\begin{equation*}
			2\left(-A_{ij}	+B_{i} \right) h_{iii}h_{jji}
			\geq -a_{ij}h_{iii}^{2}-b_{ij}h_{jji}^{2}\, ,
		\end{equation*}
		where $a_{ij}>0, b_{ij}>0$ are constants satisfying  
		\begin{equation}\label{ab}
			a_{ij} b_{ij}=\left(-A_{ij}+B_{i} \right)^{2}.
		\end{equation}
		
		And
		\begin{equation*}
			\left(-A_{ij}+B_{p} \right)h_{iip}h_{jjp}
			\geq -c_{ijp}h_{iip}^{2}-d_{ijp}h_{jjp}^{2}\, ,
		\end{equation*}
		where $c_{ijp}>0$, $d_{ijp}>0$ are constants satisfying
		\begin{equation}\label{cd}
			4c_{ijp}d_{ijp}=\left(-A_{ij}+B_{p} \right)^{2}
		\end{equation}
		and $c_{ijp}=c_{jip}$, $d_{ijp}=d_{jip}$. 
		
		Thus we obtain
		\begin{align*}
			&~\sum_{i\neq j}\sum_p A_{ij} (h_{ijp}^{2}-h_{iip}h_{jjp})
			+\sum_{p}B_{p}(\nabla_{p}\sigma_{1})^{2}\\
			&\geq \sum_{i}\left(B_{i}-\sum_{j\neq i}a_{ij} \right) h_{iii}^{2}
			+\sum_{i\neq j}\Big(2A_{ij}+B_{j}-b_{ji}-\sum_{p\neq i, p\neq j}(c_{ipj}+d_{pij}) \Big)h_{iij}^{2}\\
			&~+\sum_{\neq}A_{ij}h_{ijp}^{2}.
		\end{align*}
		
		Condition \ref{condn} implies $B_{i}>0$, then we can choose $a_{ij}=\frac{1}{n-1}B_{i}$. Then, from (\ref{ab}), we have 
		$$b_{ij}=\left(-A_{ij}+B_{i} \right)^{2}a_{ij}^{-1}
		=\frac{(n-1)\left(-A_{ij}+B_{i} \right)^{2}}{B_{i}}.$$
				
		And, we can choose $c_{ijp}=d_{ijp}$, because $h_{iip}$ and $h_{jjp}$ are the same type of terms.	Furthermore, from $(\ref{cd})$, we obtain
		$$c_{ipj}=d_{pij}=\frac{1}{2}|-A_{ip}+B_{j}|.$$
				
		Then, we just need 
		\begin{equation*}
				2A_{ij}+B_{j}\geq \frac{(n-1)\left(-A_{ij}+B_{j}\right)^{2}}{B_{j}}
				+\sum_{p\neq i, p\neq j}|-A_{ip}+B_{j}|.
		\end{equation*}
				
		For $B_{j}>0$, the above inequality is equivalent to
		\begin{equation}\label{last}
			\frac{2A_{ij}}{B_{j}}+1\geq (n-1)\left(-\frac{A_{ij}}{B_{j}}+1\right)^{2}
			+\sum_{p\neq i, p\neq j}|-\frac{A_{ip}}{B_{j}}+1|.
		\end{equation} 
		It is easy to check this inequality holds if $1\leq\frac{A_{ij}}{B_{p}}\leq 1+\delta$  with $\delta$ satisfies \eqref{eq:delta} for all $1\leq p\leq n$ and $1\leq i<j\leq n$.
	\end{proof}
	
	\begin{proof}[Proof of Theorem 1.4]
		The proof is completed by the maximum principle. The equation \eqref{2-3} is elliptic and at the maximum point of $G$, the left hand side of \eqref{2-3} is non-positive. But, under Condition \ref{condn}, we know the right hand side of \eqref{2-3} is non-negative. This means TERM I must be zero. By Lemma \ref{Lem 2.1}, we obtain $\lambda_1=\lambda_2=\cdots=\lambda_n$. By Newton-Maclaurin inequality, we know $G=\frac{\sigma_{1}^{k}}{\sigma_{k}}$ also reaches its minimum, therefore is a constant. So, $\lambda_1=\lambda_2=\cdots=\lambda_n$ is established everywhere on $M$ which implies $M$ is a round sphere.
	\end{proof}
	
\section{ For $F=\sum_{l=1}^{n}a_{l}\sigma_{l}$}
\label{sec:F non-homogeneous}

	In this section, for a non-homogeneous function $F=\sum_{l=1}^{n}a_{l}\sigma_{l}$, where $a_{l}$ is a nonnegative constant and $\sum_{l=2}^{n}a_{l}>0$, we choose  $G=\frac{\sigma_{1}^{n}}{\sigma_{n}}$ as the test function. We will analyze TERM I and TERM II as in the previous section.
	
	\begin{lem}
		For $F=\sum_{l=1}^{n}a_{l}\sigma_{l}$ and $G=\frac{\sigma_{1}^{n}}{\sigma_{n}}$, TERM I is non-negative. Moreover, it vanishes if $\lambda_1=\lambda_2=\cdots=\lambda_n$.
	\end{lem}
	
	\begin{proof}
		Just need to notice that $G=\frac{\sigma_{1}^{n}}{\sigma_{n}}$ is homogeneous of degree $0$ and $\frac{\partial f}{\partial \lambda_{i}}\lambda_{i}-f=\sum_{l=2}^{n}(l-1)a_{l}\sigma_{l}>0$, the rest of the proof is similar to Lemma \ref{Lem4.1}.
	\end{proof}
	
	Now, we regard TERM II as an operator on $C^{\infty}(M)$, i.e., 
	$$\text{TERM II}=\Phi(f,g),$$ where $\Phi:C^{\infty}(M)\times C^{\infty}(M)\rightarrow C^{\infty}(M)$. Obviously, $\Phi(\sum_{l=1}^{n}a_{l}\sigma_{l},g)=\sum_{l=1}^{n}a_{l}\Phi(\sigma_{l},g)$. Now, we consider $\Phi(\sigma_{l},g)$.
	
	\begin{lem}
		For $G=\frac{\sigma_{1}^{n}}{\sigma_{n}}$, if $G$ attains its maximum at $x_{0}$, then, at $x_{0}$,
		\begin{align*}
		\Phi(\sigma_{l},g)&=g\Bigg(\sum_{i}\sum_{p\neq q}\Big(\frac{\sigma_{l-1}(\lambda|i)}{\lambda_{p}\lambda_{q}}+\frac{n\sigma_{l-2}(\lambda|pq)}{\sigma_{1}}-\frac{\sigma_{l-2}(\lambda|pq)}{\lambda_{i}})(h_{pqi}^{2}-h_{ppi}h_{qqi})\\
		&~+\sum_{i} \frac{n(n-1)\sigma_{l-1}(\lambda|i)}{\sigma_{1}^{2}}(\nabla_{i}\sigma_{1})^{2}).
		\end{align*}
	\end{lem}
	
	\begin{proof}
		By \eqref{g1} and \eqref{g-}, we have
		\begin{equation*}
		\frac{\partial g}{\partial\lambda_{p}}=g\left(\frac{n}{\sigma_{1}}-\frac{1}{\lambda_{p}}\right)
		\end{equation*}
		and
		\begin{equation*}
		\frac{\frac{\partial g}{\partial \lambda_{p}}-\frac{\partial g}{\partial \lambda_{q}}}{\lambda_{p}-\lambda_{q}}=\frac{g}{\lambda_{p}\lambda_{q}}.
		\end{equation*}
		
		Since $G$ attains its maximum at $x_{0}$, $\nabla_{l}G=0$, which implies
		$\displaystyle\frac{n\nabla_{l}\sigma_{1}}{\sigma_{1}}=\frac{\nabla_{l}\sigma_{n}}{\sigma_{n}}$ at $x_{0}$.
		Thus, at $x_{0}$, $$\frac{\partial^{2} g}{\partial\lambda_{p}\partial\lambda_{q}}h_{ppi}h_{qqi}=g\Big(\sum_{i}\frac{n(n-1)(\nabla_{i}\sigma_{1})^{2}}{\sigma_{1}^{2}}-\sum_{i}\sum_{p\neq q}\frac{1}{\lambda_{p}\lambda_{q}}h_{ppi}h_{qqi}\Big).$$
		
		Furthermore, we obtain
		\begin{align*}
			\Phi(\sigma_{l},g)&=g\Bigg(\sum_{i}\sum_{p\neq q}\Big(\frac{\sigma_{l-1}(\lambda|i)}{\lambda_{p}\lambda_{q}}+\frac{n\sigma_{l-2}(\lambda|pq)}{\sigma_{1}}-\frac{\sigma_{l-2}(\lambda|pq)}{\lambda_{i}}\Big)(h_{pqi}^{2}-h_{ppi}h_{qqi})\\
			&~+\sum_{i}\frac{n(n-1)\sigma_{l-1}(\lambda|i)}{\sigma_{1}^{2}}(\nabla_{i}\sigma_{1})^{2}\Bigg).
		\end{align*}
	\end{proof}
	
	For convenience, let 
	$$A_{ijp}=\frac{\sigma_{l-1}(\lambda|i)}{\lambda_{p}\lambda_{q}}+\frac{n\sigma_{l-2}(\lambda|pq)}{\sigma_{1}}-\frac{\sigma_{l-2}(\lambda|pq)}{\lambda_{i}}$$ 
	and 
$$B_{p}=\frac{n(n-1)}{\sigma_{1}^{2}}\sigma_{l-1}(\lambda|p).$$ 
Then,
	$$\Phi(\sigma_{l},g)=g\big(\sum_{i\neq j}\sum_p A_{ijp}(h_{ijp}^{2}-h_{iip}h_{jjp})+\sum_{p}B_{p}(\nabla_{p}\sigma_{1})^{2}\big).$$
	
	\begin{lem}\label{lem5.3}
		Let $M$ be a strictly convex hypersurface in $\mathbb{R}^{n+1}$ satisfying the condition $\lambda_{min}\geq \theta(l,n)\lambda_{max}$, where $0<\theta(l,n)\leq 1$ is a constant depending on $l$ and $n$. For $G=\frac{\sigma_{1}^{n}}{\sigma_{n}}$, if $G$ attains its maximum at $x_{0}$, then, at $x_{0}$, $\Phi(\sigma_{l},g)$ is non-negative.
	\end{lem}
	
	\begin{proof}
		Similar to the proof of Lemma \ref{Lem 4.3}, we obtain
		\begin{align*}
			&~\sum_{i\neq j} \sum_p A_{ijp}(h_{ijp}^{2}-h_{iip}h_{jjp})+\sum_{p}B_{p}(\nabla_{p}\sigma_{1})^{2}\\
			&=\sum_{i}B_{i}h_{iii}^{2}+\sum_{i\neq j}(2A_{iji}+B_{j})h_{iij}^{2}+\sum_{\neq}A_{ijp}h_{ijp}^{2}+2\sum_{i\neq j}(B_{i}-A_{iji})h_{iii}h_{jji}\\
			&~+\sum_{\neq}(B_{p}-A_{ijp})h_{iip}h_{jjp}.
		\end{align*}
		
		We first estimate the lower bounds of the last two terms. As the previous section,
		
		\begin{equation*}
		2\left(B_{i}-A_{iji} \right) h_{iii}h_{jji}
		\geq -a_{ij}h_{iii}^{2}-b_{ij}h_{jji}^{2}\, ,
		\end{equation*}
		where $a_{ij}>0, b_{ij}>0$ are constants satisfying
		\begin{equation}\label{ab_5}
		a_{ij} b_{ij}=\left(B_{i}-A_{iji} \right)^{2}.
		\end{equation}
		
		And
		\begin{equation*}
		\left(B_{p}-A_{ijp} \right)h_{iip}h_{jjp}
		\geq -c_{ijp}h_{iip}^{2}-d_{ijp}h_{jjp}^{2}\, ,
		\end{equation*}
		where $c_{ijp}>0$, $d_{ijp}>0$ are constants satisfying
		\begin{equation}\label{cd_5}
		4c_{ijp}d_{ijp}=\left(B_{p}-A_{ijp} \right)^{2},
		\end{equation}
		and $c_{ijp}=c_{jip}$, $d_{ijp}=d_{jip}$.
		
		Then, we obtain
		\begin{align*}
		&~\sum_{i\neq j}\sum_p A_{ijp} (h_{ijp}^{2}-h_{iip}h_{jjp})+\sum_{p}B_{p}(\nabla_{p}\sigma_{1})^{2}\\
		&\geq \sum_{i}\left(B_{i}-\sum_{j\neq i}a_{ij} \right) h_{iii}^{2}+\sum_{i\neq j}\Big(2A_{iji}+B_{j}-b_{ji}-\sum_{p\neq i,p\neq j}(c_{ipj}+d_{pij}) \Big)h_{iij}^{2}\\
		&~+\sum_{\neq}A_{ijp}h_{ijp}^{2}.
		\end{align*}
		
		We can choose $a_{ij}=\frac{1}{n-1}B_{i}$. Then, From (\ref{ab_5}), we have 
		$$
		b_{ij}=\left(-A_{iji}+B_{i} \right)^{2}a_{ij}^{-1}=\frac{(n-1)\left(-A_{iji}+B_{i} \right)^{2}}{B_{i}}.
		$$
		
		And, we can choose $c_{ijp}=d_{ijp}$, because $h_{iip}$ and $h_{jjp}$ are the same type of terms.
		Furthermore, from $(\ref{cd_5})$, we can take
		$$c_{ipj}=d_{pij}=\frac{1}{2}|-A_{ipj}+B_{j}|.$$
		
		Then, we just need 
		\begin{equation*}
		2A_{iji}+B_{j}\geq \frac{(n-1)\left(-A_{iji}+B_{j}\right)^{2}}{B_{j}}
		+\sum_{p\neq i, p\neq j} |-A_{ipj}+B_{j}|.
		\end{equation*}
		For $B_{j}>0$, the above inequality is equivalent to
		\begin{equation}\label{last_5}
		\frac{2A_{iji}}{B_{j}}+1\geq (n-1)\left(-\frac{A_{iji}}{B_{j}}+1\right)^{2}
		+\sum_{p\neq i, p\neq j}|-\frac{A_{ipj}}{B_{j}}+1|.
		\end{equation}
		
		Notice that $\frac{A_{ijp}}{B_{q}}=\frac{n}{n-1}$ at umbilical points of $M$ for any $1\leq i, j, p, q\leq n$ .
		Thus we can assume 
		\begin{equation}\label{delta}
		1<\frac{A_{ijp}}{B_{q}}<1+\delta.
		\end{equation}
		Then, by solving $$ 3\geq (n-1)\delta^{2}+(n-2)\delta, $$ we can choose $\delta=\frac{\sqrt{n^{2}+8n-8}+2-n}{2(n-1)}$ such that (\ref{last_5}) holds.
		By direct calculation, we can choose
		\begin{equation*}
		\theta(l,n) = \left\{ 
		\begin{aligned}
		\noindent \max(\sqrt{\frac{n-1}{n}},\sqrt{\frac{n}{(n-1)(1+\delta)}}),\qquad\qquad\qquad\qquad\qquad\qquad &\text{for}~ l=1,\\
		\max(\left(\frac{\frac{n-1}{n}C_{n-1}^{l-1}+C_{n-2}^{l-2}}{C_{n-1}^{l-1}+C_{n-2}^{l-2}}\right)^{\frac{1}{l+1}},
		\left(\frac{C_{n-1}^{l-1}+C_{n-2}^{l-2}}{\frac{(1+\delta)(n-1)}{n}C_{n-1}^{l-1}+C_{n-2}^{l-2}} \right)^{\frac{1}{l+1}}),&\text{for}~ l=2,...,n,
		\end{aligned}
		\right.
		\end{equation*}
		such that under the condition $\lambda_{min}>\theta(l,n)\lambda_{max}$, (\ref{delta}) holds,
			where $C_n^k=\frac{n!}{k!(n-k)!}$.
		Thus the proof is completed.	
	\end{proof}
		
	\begin{proof}[Proof of Theorem \ref{Thm1.7}]
		Now, let $\displaystyle\Theta(n)=\max_{l=1,...,n}\theta(n,l)$. Then, under condition $\lambda_{min}>\Theta\lambda_{max}$, $\Phi(\sigma_{l},g)$ is non-negative for all $l$. Therefore, TERM II is non-negative under the condition. Similarly, by the maximum principle we complete the proof.
	\end{proof}

\section{Proof of Theorem \ref{inthm}}
\label{sec:other}

	\begin{proof}[Proof of Theorem \ref{inthm}]
		By Minkowski identity, we have
		$$k\int_{M}\sigma_{k}\langle X,e_{n+1} \rangle d\mu+(n-k+1)\int_{M}\sigma_{k-1}d\mu=0.$$
		By \eqref{1-1}, we have
		\begin{equation}\label{minkv}
		0=\int_{M}k\sigma_{k}\left(-F+\frac{(n-k+1)\sigma_{k-1}}{k\sigma_{k}}\right)d\mu.
		\end{equation}
		Since $\sigma_{k}>0$ and $-F+\frac{(n-k+1)\sigma_{k-1}}{k\sigma_{k}}$ is non-negative or non-positive, we know $F=\frac{(n-k+1)\sigma_{k-1}}{k\sigma_{k}}$.
		Notice that \eqref{minkv} also holds for $k-1$. Combining
		 Newton-MacLaurin inequalities, we have
		\begin{align*}
		0&=\int_{M}(k-1)\sigma_{k-1}\left(-F+\frac{(n-k+2)\sigma_{k-2}}{(k-1)\sigma_{k-1}}\right)d\mu\\
		&=\int_{M}(k-1)\sigma_{k-1}\left(-\frac{(n-k+1)\sigma_{k-1}}{k\sigma_{k}}+\frac{(n-k+2)\sigma_{k-2}}{(k-1)\sigma_{k-1}}\right)d\mu\leq 0.
		\end{align*}
		This implies
		$$\frac{(n-k+1)\sigma_{k-1}}{k\sigma_{k}}=\frac{(n-k+2)\sigma_{k-2}}{(k-1)\sigma_{k-1}}$$ on $M$. Thus, we have $\lambda_{1}=\lambda_{2}=\cdots=\lambda_{n}$ for every point of $M$ which means $M$ is a round sphere.
	\end{proof}
	

	\begin{acknow}
		The authors would like to thank Professor Xinan Ma for his nice lectures on $\sigma_k$-problems delivered in Tsinghua University in January 2016.
		They would also thank Professor Haizhong Li for his valuable comments.
		\end{acknow}


\end{document}